\documentclass[11pt]{article}
\usepackage{latexsym,amsmath,amssymb}

\newif\ifdviwin

\usepackage[center]{titlesec}
\usepackage[english]{babel}
\usepackage{indentfirst}
\usepackage[mathscr]{eucal}
\usepackage{amssymb,amsmath,amsfonts}
\usepackage{fancybox,fancyhdr}
\usepackage{graphicx}
\usepackage[utf8]{inputenc}
\usepackage{float}
\usepackage{color}
\usepackage[pdftex]{hyperref}
\hypersetup{colorlinks=true,linkcolor=blue,citecolor=blue} 

\newif\ifdviwin

\dviwintrue

\def\H{\mathcal{H}}
\def\H{\mathfrak{h}}

\def\m2r{\mathbb{M}^2\times\mathbb{R}}
\def\h2r{\mathbb{H}^2\times\mathbb{R}}

\let\infty=\infty \let\0=\emptyset

\let\tilde=\widetilde

\let\parc=\partial
\let\varepsilon=\varepsilon

\def\cte.{\mathop{\rm cte.}\nolimits}

\def\N{\mathbb{N}}

\def\R{\mathbb{R}}

\def\M{\mathbb{M}}

\def\H{\mathcal{H}}

\def\D{\mathbb{D}}

\def\S{\mathbb{S}}
\def\X{\mathfrak{X}}

\def\m2r{\M^2\times\R}
\def\mkr{\M^2(\kappa)\times\R}

\def\h2r{\mathbb{H}^2\times\R}
\def\s2r{\mathbb{S}^2\times\R}
\def\hkr{\mathbb{H}^2(\kappa)\times\R}

\def\nil{\mathrm{Nil}_3}
\def\psl{\tilde{PSL_2}(\R)}
\def\hipar{\H\in \mathfrak{C}^1_{\kappa,\mathrm{even}}(I)}
\def\hi{\H\in \mathfrak{C}^1_\kappa(I)}
\def\hh{\H\in C^1([-1,1])}
\def\Hs{\mathcal{H}\text{-}\mathrm{surface}}
\def\Hss{\mathcal{H}\text{-}\mathrm{surfaces}}

\def\Hgg{\mathcal{H}\text{-}\mathrm{graphs}}
\def\sig{\Sigma}
\def\r3{\mathbb{R}^3}

 \newtheorem{defi}{Definition}[section]
 \newtheorem{teo}[defi]{Theorem}
 \newtheorem{pro}[defi]{Proposition}
 \newtheorem{cor}[defi]{Corollary}
 \newtheorem{lem}[defi]{Lemma}

 \newenvironment{proof}{\rm \trivlist \item[\hskip \labelsep{\it
      Proof}:]}{\par\nopagebreak \hfill $\Box$ \endtrivlist}

 \newenvironment{proofsk}{\rm \trivlist \item[\hskip \labelsep{\it
      Sketch of the proof }:]}{\par\nopagebreak \hfill $\Box$ \endtrivlist}

\numberwithin{equation}{section}

\begin{document}


\begin{center}

\renewcommand{\thefootnote}{\,}
{\large \bf Prescribed mean curvature surfaces in the product spaces $\mathbb{M}^2(\kappa)\times\R$; Height estimates and classification results for properly embedded surfaces
\footnote{\hspace{-.75cm}
The author was partially supported by MICINN-FEDER Grant No. MTM2016-80313-P, Junta de Andalucía Grant No. FQM325 and FPI-MINECO Grant No. BES-2014-067663.}}\\
\vspace{0.5cm} { Antonio Bueno}\\
\end{center}
\vspace{.5cm}
Departamento de Geometría y Topología, Universidad de Granada, E-18071 Granada, Spain. \\ 
\emph{E-mail address:} jabueno@ugr.es \vspace{0.3cm}

\begin{abstract}
The aim of this paper is to extend classic results of the theory of CMC surfaces in the product spaces to the class of immersed surfaces in $\mkr$ whose mean curvature is given as a $C^1$ function depending on their angle function. We cover topics such as the existence of a priori curvature estimates for graphs, height estimates for horizontal and vertical compact graphs and a structure-type result, which classifies the properly embedded surfaces with finite topology and at most one end.
\end{abstract}

\section{\large Introduction}
\vspace{-.5cm}
One of the peaks in the theory of positive constant mean curvature surfaces (CMC surfaces in the following) in the Euclidean space $\r3$ is the \emph{structure theorem for properly embedded CMC surfaces} due to Meeks \cite{Mee}, which states the following: 

\emph{Let $\sig$ be a properly embedded CMC surface in $\r3$ with $k$ ends\footnote{Let $\sig$ be a properly embedded CMC surface in $\r3$. Then, $\sig$ is homeomorphic to a closed (compact without boundary) surface with $k\geq 0$ points $\{p_i\}_{i=1,...,k}$ removed, and there exists neighborhoods $U_i$ around each $p_i$ s.t. $E_i:=U_i-\{p_i\}$ is homeomorphic to $\S^1\times [0,\infty)$. Then, each $E_i$ is called an end of $\sig$.}. Then, each end of $\sig$ is cylindrically bounded, and all of them cannot lie in the same half-space. In particular:
\begin{itemize}
\item[1.] If $k=0$, then $\sig$ is the totally umbilical CMC sphere.	
\item[2.] The case $k=1$ cannot happen.
\item[3.] If $k=2$, then $\sig$ is cylindrically bounded. Moreover, $\sig$ is a right circular cylinder or an unduloid.
\item[4.] If $k=3$, then $\sig$ is contained in a \emph{slab} (the open region delimited by two parallel planes).
\end{itemize}}

This pioneer work from Meeks was also generalized to other ambient spaces, see \cite{KKMS} for an adaptation to the hyperbolic space $\mathbb{H}^3$. The cornerstone for obtaining this result was the existence of \emph{uniform height estimates} for CMC graphs, and the possibility of extending these height estimates to compact, embedded surfaces by reflecting w.r.t. totally geodesic surfaces by applying Alexandrov's reflection technique. In both spaces this procedure can be done in any direction of the space, since both are \emph{isotropic}; their isometry group act transitively in their tangent bundle, and no direction at all is in some sense \emph{privileged}.

Besides these space forms, which have isometry group with dimension 6, the product spaces $\mkr$ defined as the riemannian product of the complete simply connected surface $\M^2(\kappa)$ with constant curvature $\kappa$ and the real line $\R$, are among the next most symmetric spaces having isometry group of dimension 4. The theory of immersed surfaces in the product spaces expermiented an extraordinary development since Abresch and Rosenberg \cite{AbRo} defined a holomorphic quadratic differential on every constant mean curvature surface, that vanishes on rotational examples. This quadratic differential, called in the literature the \emph{Abresch-Rosenberg differential}, enabled the authors to extend the so called Hopf theorem: an immersion of a closed surface with genus zero and positive constant mean curvature is a rotational, embedded sphere.

Although these spaces are not isotropic, the existence of uniform height estimates as well as some results of Meeks structure theorem have been generalized when $\sig$ is either a CMC surface or a surface with positive constant extrinsic curvature in $\mkr$, see \cite{AEG,ChRo,EGR, HLR,Maz,NeRo} for instance. We should highlight Nelli and Rosenberg's work \cite{NeRo}, where the authors showed the non-existence of properly embedded surfaces with one end and constant mean curvature $H_0>1/\sqrt{3}$ by proving a distance lemma to the boundary of each compact, stable CMC surface with $H_0>1/\sqrt{3}$. The sharp height estimates obtained in \cite{AEG} ensures us that this non-existence theorem holds for every constant $H_0>1/2$. This inequality is optimal, since for the value of the mean curvature $H_0=1/2$ there exists CMC entire graphs in $\h2r$.

The class of immersed surfaces on which we will focus in this paper is the following:

\begin{defi}
Let $\H\in C^1([-1,1])$ be a $C^1$ function. An immersed surface $\sig$ in the product space $\M^2(\kappa)\times\R$ has \emph{prescribed mean curvature} $\H$ if its mean curvature $H_\sig$ satisfies
\begin{equation}\label{defHsup}
H_\sig(p)=\H(\nu_p),\hspace{.5cm} \forall p\in\sig,
\end{equation}
where $\nu:\sig\rightarrow\R$ is the \emph{angle function} of $\sig$.
\end{defi}
For short, we will say that $\sig$ is an $\H$\emph{-surface}. 

The problem of studying immersed surfaces governed by some symmetric function depending on its principal curvatures and its Gauss map has its origins in a long standing conjecture due to Alexandrov \cite{Ale} regarding the uniqueness of strictly convex spheres with \emph{prescribed Weingarten curvature} \footnote{An immersed surface $\sig$ has \emph{prescribed Weingarten curvature} if it is governed by a function of its principal curvatures and its Gauss map.} in $\R^3$. This conjecture was recently solved by Gálvez and Mira as a consequence of their oustanding work \cite{GaMi}, where the authors announced an extremely general Hopf-type theorem\footnote{In the literature, a \emph{Hopf-type theorem} refers to a uniqueness result of spheres in some class of immersed surfaces} for immersed surfaces governed by an elliptic PDE in an arbitrary oriented three-manifold.

For the particular but important case that we prescribe the mean curvature, the global theory of \emph{prescribed mean curvature surfaces} has been recently developed by Bueno, Gálvez and Mira in \cite{BGM}, where the authors covered topics such as the classification of rotational $\Hss$, including a Delaunay-type classification result; the existence of a priori curvature and height estimates for compact $\Hgg$; a classification result for properly embedded $\Hss$; the stability of $\Hss$ and a radius estimate for stable $\Hss$.

Notice that there are classes of immersed in $\m2r$ whose mean curvature is given by \eqref{defHsup} that have been previously studied. For instance, if $\H(y)=y$ then the immersed surfaces in $\mkr$ defined by \eqref{defHsup} are the \emph{self-translating solitons of the mean curvature flow} in the product spaces $\mkr$. The properties of these solitons have been studied recently in \cite{Bue1,LiMa}.

In general, the class of immersed surfaces in $\mkr$ governed by Equation \eqref{defHsup} has been already defined in \cite{Bue3}, where we studied rotational $\Hss$ and generalized some of the ideas developed in \cite{BGM}. Specifically, we studied rotational $\Hss$ by means of a phase plane analysis which led us to prove the existence and uniqueness of immersed $\H$-spheres and a Delaunay-type\footnote{In the theory of CMC surfaces in $\R^3$, the \emph{Delaunay theorem} states that there are exactly four types of complete, rotational CMC surfaces: a totally umbilical sphere, a right circular cylinder, a uniparametric family of onduloid-type surfaces and a uniparametric family of nodoid-type surfaces. In \cite{Bue3} we extended this result for the class of $\Hss$, under some hypothesis on $\H$.} classification result, under some geometric conditions on $\H$. We also focused on the construction of entire, strictly convex $\H$-graphs and properly embedded annuli, named $\H$-catenoids for resembling reasons with the minimal case $\H\equiv 0$, in the spaces $\hkr$. Finally, we analyzed further examples of $\Hss$ in $\mkr$, showing the richness and wideness of the class of $\Hss$ for arbitrary choices of $\H$.

While studying these further examples, we constructed $\Hss$ that cannot exist in the CMC theory. Indeed, we proved the existence of properly embedded disks which are cylindrically bounded, see Section 5 in \cite{Bue3}. Motivated by the
natural mixture between Meeks structure theorem,  the works of Espinar, Gálvez and Rosenberg \cite{EGR} and Nelli and Rosenberg \cite{NeRo}, the global properties of $\Hss$ studied in Section 4 in \cite{BGM}, and the singular examples arising in \cite{Bue3}, the aim of this paper is to understand the global behavior of properly embedded $\Hss$ in $\mkr$ with finite topology. 

The rest of the introduction is devoted to detail the organization of the paper, and highlight some of the main results.

In \textbf{Section \ref{secpropbasicas}} we will define 	the product spaces $\mkr$ as the riemannian product of the complete, simply connected surface $\M^2(\kappa)$ of constant curvature $\kappa$\footnote{Up to an homothety of the metric, we will assume that $\kappa=\pm 1$.} and the real line $\R$, endowed with the usual product metric. In Equation \eqref{espaciosc1h} we define the spaces of functions $\mathfrak{C}_\kappa^1(I)$ and $\mathfrak{C}_{\kappa,even}^1(I)$, that will be useful in the development of this paper. We recall the group of isometries of $\m2r$, and how the symmetries of Equation \eqref{defHsup} induce further invariance properties of the class of $\Hss$, as detailed in Lemma \ref{sime}. In particular, almost all the isometries in $\m2r$ will be also isometries for the class of $\Hss$. 

Inspired by the ideas developed in \cite{RST}, we announce a curvature estimate lemma \ref{estuniforme} for a sequence of compact $\Hgg$, and a compactness theorem \ref{compa} for the space of $\Hss$ with bounded curvature. In some sense these results will allow us to take limit in a sequence of $\Hss$ with bounded curvature and having a finite accumulation point.

In \textbf{Section \ref{secuniformheight}} we will focus in obtaining uniform height estimates for compact $\H$-graphs. In Definition \ref{uhai} we settle on the concept of uniform height estimate w.r.t. a horizontal or a vertical direction in $\m2r$. For the spaces $\h2r$, in Proposition \ref{estialturahorizontal} we obtain a uniform height estimates for horizontal compact $\Hgg$. Notice that these estimates does not make sense in the space $\s2r$ since the base is compact.

For the case where the graphs are vertical, the situation is trickier. In \cite{Bue3} we constructed, for a prescribed function $\hi$, vertical proper $\Hgg$ in $\m2r$ contained in a vertical cylinder, and thus having unbounded height to any horizontal plane. These examples show that more conditions on $\H$ are needed if we expect to obtain the desired estimates. In this situation, if $\hipar$ then in Theorem \ref{estialturavertical} we obtain a uniform height estimates for vertical compact $\Hgg$ in the spaces $\m2r$. 

They way of obtaining these estimates is completely different from the ideas used by Hoffman, de Lira and Rosenberg for CMC surfaces in \cite{HLR} and by Espinar, Gálvez and Rosenberg for positive extrinsic curvature surfaces in \cite{EGR}. As $\H$ is an even function, Alexandrov's reflection technique guarantees the existence of uniform height estimates for compact, embedded surfaces with boundary contained in either horizontal or vertical planes.

Lastly, in \textbf{Section \ref{secestructura}} we study properly embedded $\Hss$ with finite topology. First we recall an outstanding theorem due to Meeks \cite{Mee}, the plane separation lemma, see \ref{seplem}. This theorem along with the height estimates obtained in Section \ref{secuniformheight}, allows us to prove in Proposition \ref{1finalcilac} that if $\hi$, then every properly embedded $\Hs$ in $\h2r$ with finite topology and one end, must lie inside a solid vertical cylinder. For a prescribed function $\hipar$, we obtain in Theorem \ref{no1final} a more definite result: the non-existence of properly embedded $\Hss$ with finite topology and one end in the spaces $\m2r$, extending the known results due to Nelli and Rosenberg \cite{NeRo} in $\h2r$. We conclude this paper by announcing a structure-type classification result in Theorem \ref{estructura1final}: a properly embedded $\Hs$ with finite topology and at most one end, must be a rotational $\H$-sphere.

\section{\large Basic properties of $\H$-surfaces in the spaces $\mkr$}\label{secpropbasicas}
\vspace{-.5cm}

Let $\M^2(\kappa)$ be the complete, simply connected surface with constant curvature $\kappa$. Then, up to isometries, $\M^2(\kappa)$ is one of the following surfaces: if $\kappa=0$ we get the usual flat plane $\R^2$; if $\kappa<0$ we obtain the hyperbolic plane $\mathbb{H}^2(1/\sqrt{-\kappa})$; and if $\kappa>0$ we recover the totally umbilical sphere $\S^2(1/\sqrt{\kappa})$. We drop out the case $\kappa=0$, since the theory of immersed $\Hss$ in $\R^3$ was widely studied in \cite{BGM}. These non-flat surfaces can be regarded as isometrically immersed surfaces in the space $\R^3_\kappa$, where $\R^3_\kappa$ stands for the usual Euclidean space if $\kappa>0$ or for the Lorentz-Minkowski space $\mathbb{L}^3$ endowed with the metric with signature $+,+,-$ if $\kappa<0$. Indeed, the surface $\M^2(\kappa)$ can be defined as the quadric
\begin{equation}\label{modelobase}
\M^2(\kappa)=\{(x_1,x_2,x_3)\in\R^3_\kappa;\ x_1^2+x_2^2+\kappa x_3^2=1/\kappa,\ (1-\kappa)x_3>0\}.
\end{equation}
Clearly, after an homothety of the metric we can suppose that $\kappa=\pm 1$, and we drop the dependence on $\kappa$ by just writing $\M^2$. 

The product spaces $\m2r$ are defined as the riemannian product of the complete surface $\M^2$ with the real line, endowed with the product metric. This product structure ensures the existence of a Killing vector field defined as the gradient of the projection onto the real factor, called the \emph{vertical vector field} and denoted by $e_3$. The \emph{vertical planes} are the surfaces given by $\gamma\times\R$, where $\gamma\subset\M^2$  is a geodesic, and they are totally geodesic surfaces isometric to $\R^2$. The \emph{horizontal planes} are the surfaces given by $\M^2\times\{t_0\},\ t_0\in\R$, and they are totally geodesic surfaces isometric to $\M^2$. We may refer also to these horizontal planes as \emph{slices}. Commonly, we will identify the base $\M^2$ with the totally geodesic slice $\M^2\times\{0\}$.

For the possible values of $\kappa=\pm 1$, we define the sets of functions
\begin{equation}\label{espaciosc1h}
\begin{array}{c}
\mathfrak{C^1_\kappa}(I)=\Big\{\H\in C^1(I);\ 4\H(y)>(1-\kappa),\ \forall y\in I\Big\},\\

\mathfrak{C}^1_{\kappa,\mathrm{even}}(I)=\Big\{\H\in C^1(I);\ \H(y)=\H(-y),\ 4\H(y)>(1-\kappa)\sqrt{1-y^2},\ \forall y\in I \Big\},
\end{array}
\end{equation}
where $I$ will denote throughout this paper the closed interval $[-1,1]$.

For the particular case that $\H=H_0\in\R$, this condition has a strong geometric sense in the space $\h2r$. Indeed, if $\kappa=-1$ the value $1/2$ is \emph{critical} for the existence of CMC spheres; this condition ensures us the existence of a sphere with constant mean curvature $H_0<\min\H$. In our paper \cite{Bue3} we related this hypothesis with the existence of $\H$-spheres, see Theorem 3.8.

Besides the space forms $\R^3,\mathbb{H}^3$ and $\S^3$, whose isometry group has dimension six (the highest for a three dimensional space), the product spaces $\m2r$ have isometry group of dimension four, the second highest for a three-dimensional space. This group is generated by three linearly independent translations and the isometric $SO(2)$ action of rotations that leaves pointwise fixed a \emph{vertical line} of the form $p\times\{\R\}$, where $p\in\M^2$. 

One of the uniparametric group of translations is the flow of the vertical Killing vector field $e_3$. The product structure ensures us that the other two linearly independent group of translations in $\m2r$ are generated by two linearly independent translations of the base $\M^2$. The remaining isometries in $\m2r$ are obtained by linear combinations of the previously introduced isometries.

Next, we derive a consequence of Equation \eqref{defHsup} when we see an $\Hs$ locally as a vertical or horizontal graph. Let $\sig$ be an $\Hs$ and take some $p\in\sig$. Suppose that $\eta_p$ is not an horizontal vector. Thus, the implicit function theorem ensures us that a neighborhood of $p\in\sig$ can be expressed as a vertical graph of a function $u:\Omega\subset\M^2\rightarrow\R$. In this situation, Equation \eqref{defHsup} has the following divergence expression
$$
{\rm div_{\M^2}}\left(\displaystyle \frac{\nabla^{\M^2} u}{\sqrt{1+|\nabla^{\M^2} u|^2}}\right) = 2\H \left(\displaystyle\frac{1}{\sqrt{1+|\nabla^{\M^2}u|^2}}\right),
$$
where ${\rm div_{\M^2}}$ and $\nabla^{\M^2}$ are the divergence and gradient operators, both computed w.r.t. the metric of the base. If $\eta_p$ is horizontal, then $\sig$ can be expressed as a horizontal graph which also satisfies a divergence-type equation, see e.g. Section 5 in \cite{Maz}. In particular, the $\Hss$ satisfies the Hopf maximum principle in its interior and boundary versions, a property that has the following geometric implication

\begin{lem}[\textbf{Maximum principle for $\H$-surfaces}]\label{pmax}
Given $\hh$, let $\Sigma_1,\Sigma_2$ be two $\H$-surfaces, possibly with non-empty smooth boundary. Assume that one of the following two conditions holds:
\begin{enumerate}
\item
There exists $p\in {\rm int}(\Sigma_1)\cap {\rm int}(\Sigma_2)$ such that $\eta_1(p)=\eta_2(p)$, where $\eta_i$ is the unit normal of $\Sigma_i$, $i=1,2$.
\item
There exists $p\in \partial \Sigma_1 \cap \partial \Sigma_2$ such that $\eta_1(p)=\eta_2(p)$ and $\xi_1(p)=\xi_2(p)$, where $\xi_i$ denotes the interior unit conormal of $\partial \Sigma_i$.
\end{enumerate}
Assume moreover that $\Sigma_1$ lies around $p$ at one side of $\Sigma_2$. Then $\Sigma_1=\Sigma_2$.
\end{lem}

This geometric maximum principle is commonly used along with the isometries of the ambient space. Given $\hh$, $\sig$ an $\Hs$ and $\Phi$ an isometry of $\m2r$, if we ask $\Phi(\sig)$ to be an $\Hs$ for the same prescribed function $\H$, then by Equation \eqref{defHsup} $\Phi$ must keep invariant the angle function $\nu$ of $\sig$. 

For example, the translations of an immersed surface in $\m2r$ does not change the value of the coordinates of its unit normal, and thus it is immediate that any translation of an $\Hs$ is again an $\Hs$. Moreover, almost all the isometries of $\m2r$ will be also isometries for the class of immersed $\Hss$, as a straightforward consequence of the next lemma:

\begin{lem}\label{sime}
Given $\hh$, let be $\Sigma$ an $\H$-surface, $\eta:\Sigma\rightarrow T\m2r$ a unit normal vector field and $\Phi$ an isometry of $\m2r$ such that $\H \circ \Phi (y)=\H(y),\ \forall y\in [-1,1]$.
Then $\Sigma'=\Phi \circ \Sigma$ is an $H$-surface in $\m2r$ with respect to the orientation given by $\eta':=d\Phi(\eta)$.
\end{lem}

In particular, as the angle function $\nu$ on each surface is invariant under the rotations around any vertical axis, Equation \eqref{defHsup} implies that the class of $\Hss$ contains the rotations around each vertical line as a uniparametric group of isometries. As a matter of fact, the class of $\Hss$ also contains the reflections with respect to vertical planes as isometries. For the particular case that $\H$ is an even function, the reflections w.r.t. horizontal planes are also isometries; these reflections changes the sign of $\nu$ but $\H$ takes the same value on $\nu$ and $-\nu$, inducing the corresponding reflection as an isometry.

The last two results in this section will be the cornerstone in the proof of further results. Both are inspired by the ideas developed in the proof of the main theorem in \cite{RST}, where the authors obtained a general bound for the second fundamental form and the diameter of an immersed CMC surface. These results were also extended to the case of $\Hss$ in $\r3$ \cite{BGM} and for CMC surfaces in the homogeneous spaces $\nil$ and $\psl$ \cite{Bue2}. The proof of both results are quite similar to the ones exposed in \cite{Bue2,BGM}, and this is the reason for omitting them.

The first one is a uniform curvature estimate for a sequence of compact $\Hgg$, which ensure that those graphs cannot have \emph{blowing second fundamental form}, or equivalently, blowing curvature. 
\begin{lem}\label{estuniforme}
Let be $\hh,\ \Pi$ a horizontal plane and $\{\sig_n\}$ a sequence of compact vertical $\H$-graphs in $\m2r$, with $\partial \sig_n\subset\Pi$ and with heights to $\Pi$ tending to infinity, and fix $d>0$. Then, there exists a positive constant $\Lambda$ such that the uniform estimate holds:
$$
|\sigma_{\sig_n}(p)|<\Lambda,\ \forall p\in \sig_n^*,\ \forall n\in\N,
$$
where $\sigma_{\sig_n}$ denotes the second fundamental form of each $\sig_n$ and $\sig_n^*$ is the subset of $\sig_n$ defined as
$$
\sig_n^*=\{p\in \sig_n;\ d(p,\partial \sig_n)>2d\}.
$$
\end{lem}

This curvature estimate implies the following \emph{compactness theorem} for the space of $\Hss$ in $\m2r$ with bounded second fundamental form. In some sense, this result allows us to take limit in a sequence of $\Hss$ with bounded curvature, which has some limit point in $\m2r$ as an accumulation point. 

\begin{teo}\label{compa}
Let $(\Sigma_n)_n$ be a sequence of $\H_n$-surfaces in $\m2r$ for some sequence of functions $\H_n\in C^k([-1,1])$, $k\geq 1$, and take $p_n\in \Sigma_n$. Assume that the following conditions hold:
 \begin{enumerate}
\item[1.] There exists a sequence of positive numbers $r_n\to\infty $ such that the geodesic disks $D_n= D_{\Sigma_n} (p_n, r_n)$ are contained in the interior of $\Sigma_n$, i.e. $d_{\Sigma_n} (p_n,\parc \Sigma_n) > r_n$.
 
\item[2.] The sequence $(p_n)_n $ has some $p_\infty\in\m2r$ as accumulation point.

\item[3.] If $|\sigma_n |$ denotes the length of the second fundamental form of $\Sigma_n$, then there exists a constant $C>0$ such that $| \sigma_n | (x) \leq C$ for every $n$ and every $x \in \Sigma_n$.

\item[4.] $\H_n \to \H_\infty$ in the $C^k$ topology to some $\H_\infty\in C^1([-1,1])$.
\end{enumerate}
Then, there exists a subsequence of $(\Sigma_n)_n$ that converges uniformly on compact sets in the $C^{k+2}$ topology to a complete, possibly non-connected, $\H_\infty$-surface $\Sigma_\infty$ of bounded curvature that passes through $p_\infty$.
\end{teo}

\section{\large A uniform height estimate for compact graphs}\label{secuniformheight}
\vspace{-.5cm}

In the next definition we formalize the notion of height estimate in the class of prescribed mean curvature graphs.

\begin{defi}\label{uhai}
Let be $\hh$, and suppose that $v$ is either a horizontal or a vertical unit vector. We will say that there exists a \emph{uniform height estimate for $\H$-graphs in the $v$-direction} if there exists a constant $C(\H,v,\kappa)>0$ such that the following assertion is true:

For any graph $\Sigma$ in $\m2r$ of prescribed mean curvature $\H$ oriented towards $v$ (i.e. $\langle \eta,v\rangle>0$ on $\Sigma$ where $\eta$ is the unit normal of $\Sigma$), and with $\parc \Sigma$ contained in the plane $\Pi=v^{\bot}$, it holds that the distance of any $p\in \Sigma$ to $\Pi$ is at most $C(\H,v,\kappa)$.
\end{defi}

Note that the concept of uniform height estimate in a horizontal direction only makes sense in the space $\h2r$, since the base of the space $\s2r$ is compact and thus horizontal graphs automatically satisfy a horizontal height estimate.  For the space $\h2r$, since rotations with respect to a vertical line are isometries for the class of immersed $\Hs$, it follows that the constant $C(\H,v,\kappa)$ in Definition \ref{uhai} is independent of the choice of the horizontal direction $v$ fixed. In this situation, we may refer simply to \emph{horizontal height estimates}. 

For the case that $\H$ is an even function, the notion of vertical uniform height estimate and the constant $C(\H,\pm e_3,\kappa)$ are also independent of the direction $e_3$ or $-e_3$ chosen. Therefore, for a fixed $\hh$, the constant $C(v,\kappa)$ only depends on the space where the graph is defined and on the vertical or horizontal direction on which we define the graph. We will denote by $C_\textbf{h}$ to the constant in Definition \ref{uhai} for the class of horizontal $\Hgg$ in $\h2r$. Since we will only focus on even functions in the study of vertical height estimates, we will denote just by $C^\kappa$ to the constant in Definition \ref{uhai} if the graph is vertical.

Let us introduce some notation for the space $\h2r$, for which we will use the \emph{Poincaré disk model} of $\mathbb{H}^2$. Let $\Pi$ be a vertical plane in $\h2r$ passing through the origin \textbf{o}, and consider a complete horizontal geodesic $\sigma(t)$ in $\mathbb{H}^2\equiv\mathbb{H}^2\times\{0\}$ such that $\sigma(0)=\textbf{o}\in\Pi$ and $\sigma'(0)\bot\Pi$. The \emph{oriented foliation of vertical planes along $\sigma$} is the family of vertical planes $\Pi_t$ at oriented distance $t$ to $\Pi$, such that $\Pi_0=\Pi$ and $\sigma'(t)\bot\Pi_t,\ \forall t$. Those planes are obtained as the image of $\Pi$ under the translations corresponding to the flow generated by a horizontal Killing vector field, having $\sigma(t)$ as integral curve.

The next two results are original from the pioneer work of Meeks \cite{Mee} for CMC surfaces in the Euclidean space $\r3$. The ideas developed were also adapted for the case of surfaces in $\h2r$ with constant extrinsic curvature, see the proof of Theorem 6.2 in \cite{EGR}. Here we extend both results for the case of $\Hss$ in $\m2r$. For the first we omit its proof, and for the second we just give a sketch of it.

\begin{lem}\label{lemi}
Let be $\hi$ and $\Sigma$ an $\H$-graph in $\m2r$ defined over a closed (not necessarily bounded) domain contained in an either vertical or horizontal plane $\Pi$, with zero boundary values. Denote by $d$ to the diameter of the sphere $S_{H_0}$ of constant mean curvature $0<H_0<\min\H$ in $\m2r$. Then, for every $t>2d$, the diameter of each connected component of $\Sigma\cap \Pi_t$ is at most $2d$.
\end{lem}

\begin{pro}\label{estialturahorizontal}
Given $\H\in \mathfrak{C}^1_{-1}(I)$, there exists a uniform height estimate for horizontal, compact $\H$-graphs in $\h2r$.
\end{pro}

\begin{proofsk}
Arguing by contradiction, assume that such estimate does not hold. Then, there exists a sequence of horizontal, $\Hgg$ $\sig_n$ over domains $\Omega_n$ which can be supposed to lie in the same vertical plane $\Pi_w$ where their boundaries $\partial\sig_n=\partial\Omega_n$ lie, and such that the maximum distance from $\sig_n$ to $\Pi_w$ tends to infinity. Moreover, after intersecting the $\Hgg$ with a plane parallel to $\Pi_w$ at distance big enough, Lemma \ref{lemi} ensures us that the domains $\Omega_n$ satisfy a uniform diameter estimate, and thus they can be supposed to lie inside a bounded domain $\Omega_\infty$. We keep naming these intersected graphs by $\sig_n$.

In this situation, all the $\Hgg$ lie in the interior of the \emph{horizontal Killing cylinder} $\Omega_\infty(t)$. Let $v$ be a unitary horizontal vector which makes an angle of $\pi/4$ with $w$. As mentioned in Lemma \ref{sime}, we can apply Alexandrov reflection technique with respect to the family of vertical planes $\Pi_v$ orthogonal to $v$. 

Now, arguing as in Theorem 6.2 in \cite{EGR} making reflections with respect to planes $\Pi_v(t)$ orthogonal to that \emph{tilted direction} $v$, we ensure the existence of a positive constant $C_\textbf{h}=C(\H,-1)>0$, which is independent of $\sig$, such that the distance of each $p\in\sig$ to $\Pi_w$ is bounded by $C$. This proves Proposition \ref{estialturahorizontal}
\end{proofsk}

As we can reflect with respect to any vertical plane, the Alexandrov reflection technique allows us to extend the previous uniform height estimates for compact embedded $\Hss$ with boundary contained in a vertical plane.

\begin{cor}\label{estialturahorizontalcompac}
Given $\H\in \mathfrak{C}^1_{-1}(I)$, every compact, embedded $\H$-surface in $\h2r$ with boundary contained in a vertical plane $\Pi$ has distance to $\Pi$ bounded by $2C_\textbf{h}$, where $C_\textbf{h}$ is the constant appearing in the proof of Proposition \ref{estialturahorizontal}.
\end{cor}
The cornerstone in the proof of Proposition \ref{estialturahorizontal} is that the symmetries in Equation \ref{defHsup} induced by $\H$ allows us to reflect with respect to \emph{tilted} planes to the direction of which $\sig$ is a graph, and the fact that $\sig$ can be supposed to be cylindrically bounded from a fixed height. 

When trying to obtain height estimates for vertical $\H$-graphs, we observe that they can be also supposed to be cylindrically bounded by means of Lemma \ref{lemi}, but there do not exist totally geodesic surfaces which are tilted with respect to the vertical direction. Thus, we need to develop new ideas in order to obtain vertical height estimates. The following result is inspired by the proof of Theorem 4.5 in \cite{BGM}. 

\begin{teo}\label{estialturavertical}
For each $\hipar$, there exists a uniform height estimate for vertical $\H$-graphs in $\m2r$.
\end{teo}
The fact that $\H$ belongs to the space of functions $\hipar$ is key in the proof of this result, as this hypothesis guarantees the existence of an $\H$-sphere $S_\H$ in the class of $\Hss$, see Theorem 3.8 in \cite{Bue3}. Indeed, if $\H$ does not satisfy the bound $4\H(y)>(1-\kappa)\sqrt{1-y^2}$ or if $\H$ is not an even function, then in Sections 4 and 5 in \cite{Bue3} we constructed examples of $\Hss$ that unable the existence of an $\H$-sphere. In contrast with other results, we cannot simply consider a CMC sphere with mean curvature lower than $\H$.

\begin{proof}
Arguing by contradiction, suppose that $\hipar$ is a function for which there is no uniform height estimate for vertical $\H$-graphs.

In this situation, there exists a sequence $(\Sigma_n)_n$ of $\H$-graphs with respect to the $e_3$-direction, oriented towards $e_3$, and with boundary $\partial \Sigma$ contained in the totally geodesic slice $\Pi=\{z=0\}$, and points $p_n\in \Sigma_n$ such that the height of $p_n$ over $\Pi$ is greater than $n$. Note that by the mean curvature comparison theorem, $\Sigma_n\subset \{z\leq 0\}$.

Take now $z_0>2d$, where $d$ is the diameter of the $\H$-sphere $S_\H$, and denote $\Sigma_n^*:=\Sigma_n\cap \{z \leq -2z_0\}$. By hypothesis, $p_n\in \Sigma_n^*$ for $n$ large enough. By Lemma \ref{lemi}, the connected component $\Sigma_n^0$ of $\Sigma_n^*$ that contains $p_n$ is compact, and contained inside a vertical solid cylinder. 

Let $q_n\in \Sigma_n^0$ be a point of maximum height of $\Sigma_n^0$, and let $\Sigma_n^1:=\Sigma_n^0-q_n$ denote the translation of $\Sigma_n^0$ that takes $q_n$ to the origin \textbf{o} in $\m2r$. By Lemma \ref{estuniforme}, the norms of the second fundamental form of the graphs $\Sigma_n^1$ are uniformly bounded by some positive constant $\Lambda>0$ that only depends on $d$ and $||\H||_{C^1(\S^2)}$, and not on $n$. Moreover, the distances in $\m2r$ from the origin to $\partial \Sigma_n^1$ diverge to $\infty$. Now, applying Theorem \ref{compa} we deduce that, up to a subsequence, there are smooth compact sets $K_n$ of the graphs $\Sigma_n^1$, all of them containing the origin and with horizontal tangent plane at it, and that converge uniformly in the $C^2$ topology to a complete $\H$-surface $\Sigma_{\infty}$ of bounded curvature that passes through the origin.

Let us define next $\nu_{\infty}:=\langle\eta_\infty,e_3\rangle$, where $\eta_\infty$ is the unit normal of $\Sigma_\infty$. Note that $\eta_\infty({\bf o})=1$. As all the graphs $\Sigma_n^1$ are oriented towards $e_3$, we deduce that $\nu_\infty\geq 0$ on $\Sigma_\infty$. Furthermore, it a well known fact (see \cite{Dan} for instance) that $\nu$ is a solution to an elliptic equation on $\sig$ of the form
\begin{equation}\label{ecangulo}
\Delta_\sig\nu+\langle X,\nabla_\sig\nu\rangle+q=0,
\end{equation}
where $\Delta_\sig,\nabla_\sig$ denote the Laplace-Beltrami and gradient operators on $\Sigma_{\infty}$, $X\in \X(\Sigma_{\infty})$, and $q\in C^2(\Sigma_{\infty})$.  By the maximum principle for \eqref{ecangulo}, and the condition $\nu_{\infty}\geq 0$, we conclude that either $\nu_{\infty} \equiv 0$ on $\Sigma_{\infty}$ (which cannot happen since $\nu_{\infty}({\bf o})=1$), or $\nu_{\infty}>0$ on $\Sigma_{\infty}$. Therefore, $\Sigma_{\infty}$ is a local vertical graph, i.e. for every $p\in \Sigma_{\infty}$ it holds that $T_p \Sigma_{\infty}$ is not a vertical plane in $\m2r$.

Once here, and since $\Sigma_{\infty}$ is a limit of compact pieces of the graphs $\Sigma_n^1$, it is clear that $\Sigma_{\infty}$ is itself a proper graph in $\m2r$ over a domain $\Omega_\infty\subset\{z=0\}$ oriented towards $e_3$. By construction, this graph has horizontal tangent plane at the origin, it has bounded second fundamental form, and lies entirely in the closed half-space $\{z\geq 0\}$. Moreover, since each $\Sigma_n^0$ lies inside a vertical solid cylinder, we deduce that all points of $\Sigma_{\infty}$ lie at a bounded distance in $\m2r$ from the vertical axis passing through the origin, and thus the domain $\Omega_\infty$ is relatively compact.

At this point, a similar argument to the one used in the proof of Theorem 4.5 in \cite{BGM} ensures us that there exists a sequence of points $a_n\in\sig_\infty$ such that $\nu_\infty(a_n)\rightarrow 0$ and $a_n^3\rightarrow\infty$, where $a_n^3$ denotes the height of the point $a_n$, and in a way that the sequence of vertically translated graphs $\sig_\infty^n=\sig_\infty-(0,0,a_n^3)$ converges uniformly on compact sets to the complete cylinder $\mathbf{C}=\alpha\times\R$ with prescribed mean curvature $\H$, where $\alpha=\partial\Omega_\infty$ is a closed curve. Moreover, it is clear that $\Omega_\infty$ is actually the inner region bounded by $\alpha$. As $\mathbf{C}$ has constant angle equal to zero, the mean curvature $H_\mathbf{C}$ is constant and equal to $\H(0)=\kappa_\alpha/2$, where $\kappa_\alpha$ is the geodesic curvature of $\alpha$

Now we stand in position to prove Theorem \ref{estialturavertical}. As the cylinder $\mathbf{C}$ and the $\H$-sphere $S_\H$ are both $\Hss$, their mean curvatures agree in the points with the same angle function. In particular, the principal curvatures $\kappa_1,\kappa_2$ of $S_\H$ at any point $p_0$ of the equator of $S_\H$ and the geodesic curvature of $\alpha$ satisfy

\begin{equation}\label{relcurva}
2H_{S_\H}(p_0)=\kappa_1(p_0)+\kappa_2(p_0)=2\H(0)=2H_\mathbf{C}=\kappa_\alpha.
\end{equation}

Suppose, after a vertical translation, that the sphere $S_\H$ is a bi-graph over a domain $\Omega_{S_\H}$ contained in the plane $\{z=0\}$. If we denote by $\gamma=S_\H\cap\{z=0\}$ to the equator of $S_\H$, its geodesic curvature satisfies $\kappa_\gamma=\kappa_1$ and Equation \ref{relcurva} ensures us that $\kappa_\gamma<\kappa_\alpha$.

Moreover, as both curves $\gamma$ and $\alpha$ have constant geodesic curvature, we can translate $\alpha$ and $\gamma$ to the origin of $\h2r$ such that both $\alpha$ and $\gamma$ are circumferences and with $\alpha\subset\gamma$. In particular, the solid region $\Omega_\infty\times\R$ is contained in the solid region $\Omega_{S_\H}\times\R$. Denote by $S^-_{\H}=\{p\in S_\H;\ \nu_p\geq 0\}$ to the lower bi-graph of $S_\H$ with the upwards orientation and fix $\varepsilon>0$.

Recall that the cylinder $\mathbf{C}=\alpha\times\R$ was obtained as the limit of vertical translations of the proper graph $\sig_\infty$ over the domain $\Omega_\infty$, and this domain has $\alpha$ as boundary. At this point, if we suppose that $\sig_\infty$ has horizontal tangent plane at the origin, it is clear that we can consider the vertical translations $\sig_\infty^t:=\sig_\infty-te_3$ until there exists some $t_0<0$ such that $\sig_\infty^{t_0}$ and $S_\H^-$ have a first contact point. As $\partial S_\H^-=\partial\Omega_\H$ and $\partial\Omega_\infty$ was strictly contained in $\Omega_\H$, it is clear that this first contact point must be an interior one. But the half-sphere $S_\H^-$ and $\sig_\infty$ are both upwards oriented $\H$-graphs, contradicting the maximum principle for $\Hss$. This concludes the proof of Theorem \ref{estialturavertical}
\end{proof}

As the prescribed functions $\hipar$ are even, Alexandrov's reflection technique with respect to horizontal planes has as direct consequence the following corollary.

\begin{cor}\label{estialturaverticalcompac}
Given $\hipar$, every embedded $\H$-surface in $\m2r$ with boundary contained in a horizontal plane $\Pi$ has distance to $\Pi$ bounded by $2C^\kappa$, where $C^\kappa$ is the constant appearing in the proof of Theorem \ref{estialturavertical}.
\end{cor}

\section{\large A structure-type result}\label{secestructura}
\vspace{-.5cm}
The following result is a classic result in the study of properly embedded CMC surfaces in $\r3$ and was firstly proved by Meeks, see \cite{Mee}. This result has been also generalized to other ambient spaces, see Lemma 2.4 in \cite{NeRo} for the case of CMC surfaces in $\h2r$, or to other curvatures of the surface, see Theorem 7.2 in \cite{EGR}. We announce this result for an arbitrary surface in $\m2r$ with an adequate lower bound on the mean curvature, as the proof is similar to the original by Meeks.

\begin{teo}[\textbf{Plane separation lemma}]\label{seplem}
Let $\Sigma$ be a surface with boundary in $\m2r$, diffeomorphic to the punctured closed disk $\overline{\D}-\{0\}$. Assume that $\Sigma$ is properly embedded, and that its mean curvature $H_{\Sigma}$ satisfies $H_{\Sigma}(p)\geq H_0$ for every $p\in \Sigma$, and for some $H_0>(1-\kappa)/4$.

Let be $d>0$ be the diameter of the sphere $S_{H_0}$ with constant mean curvature $H_0$, $P_1,P_2$ two disjoint planes (either vertical or horizontal) in $\m2r$ at a distance greater than $2d$, and $P^+_1,P^+_2$ the two disjoint half-spaces determined by these planes. Then, all the connected components of either $\Sigma\cap P^+_1$ or $\Sigma\cap P^+_2$ are compact.
\end{teo}
The key in the proof of this lemma is to extract a big enough simply connected piece of $\sig$ and then compare with a big enough sphere (that exists by the lower bound on $H_\sig$), contradicting the mean curvature comparison theorem. In particular, the class of $\Hss$ for $\hi$ satisfy the hypothesis of Theorem \ref{seplem}.

The plane separation lemma allows us to prove the following result.

\begin{pro}\label{1finalcilac}
Let be $\hi$. Then, every properly embedded $\H$-surface in $\h2r$ with finite topology and one end lies inside a solid vertical cylinder.
\end{pro}

\begin{proof}
Let $\sig$ be a properly embedded $\Hs$ in $\h2r$ with finite topology and one end. After an ambient isometry, we may suppose that the origin \textbf{o} in $\h2r$ belongs to $\sig$. Given an horizontal vector $v$, let $\gamma_v(t)$ be the geodesic passing through \textbf{o} at the instant $t=0$ with direction $v$ and denote by $\Pi_t$ to the foliation of vertical planes along $\gamma_v$. Let $d$ be the diameter of the sphere $S_{H_0}$ with constant mean curvature $H_0$, fix $t_0>2d$ and consider the vertical planes $P_1=\Pi_v(t_0)$, $P_2=\Pi_v(-t_0)$. Clearly, $P_1,P_2$ are two disjoint vertical planes at distance to each other greater than $2d$. Moreover, we can assume that both $P_1$ and $P_2$ intersect $\sig$.

Since $\sig$ is a properly embedded surface with finite topology having only one end, we may decompose $\sig=\sig_0\cup\mathcal{A}$, where $\sig_0$ is a compact surface with boundary and $\mathcal{A}$ is a proper embedding of a disk $\S^1\cup [0,\infty)$ into $\h2r$ and such that $\partial\sig_0=\S^1\cup\{0\}$. In this setting, the plane separation lemma \ref{seplem} ensures us that either $\sig\cap P_1^+$ or $\sig\cap P_2^+$ only has compact connected components; for definiteness, say that $\sig\cap P_1^+$ has this property. By the height estimates existing due to Proposition \ref{estialturahorizontalcompac}, the surface $\sig$ is contained in the half-space $\Pi_v^+(t_0+2C_\textbf{h})$, where  $\Pi_v^+(t_0+2C_\textbf{h})$ is the half-space determined by  $\Pi_v(t_0+2C_\textbf{h})$ that does not contains the origin.

Now consider the planes $P_3=\Pi_v(t_0-4C_\textbf{h})$ and $P_4=\Pi_v(-t_0-4C_\textbf{h})$ and their disjoint half-spaces $P_3^+$ and $P_4^+$. The same arguments as before ensures us that at least one of $\sig\cap P_3^+$ or $\sig\cap P_4^+$ only has compact connected components. If this is the case for $\sig\cap P_3^+$, then the height estimates in Proposition \ref{estialturahorizontalcompac} ensures us that $\sig$ is contained in $P_3^-(t_0-2C_\textbf{h})$, which is a contradiction since $\sig\cap P_1\neq\varnothing$. Thus, $\sig\cap P_4^+$ has only compact connected components and arguing as in the previous paragraph we deduce that $\sig$ is contained in the half-space $P_4^+(-t_0-6C_\textbf{h})$

Notice that as Equation \eqref{defHsup} is rotationally symmetric, the choice of the horizontal vector $v$ is independent of this method. Hence, $\sig$ is contained in a solid cylinder, finishing the proof of Theorem \ref{1finalcilac}.
\end{proof}

The previous result only makes sense in the space $\h2r$ since the base of the space $\s2r$ is compact, and every properly embedded surface in $\s2r$ lies at bounded distance to any vertical line $p\times\{\R\}$.

If we suppose that $\hipar$, i.e. $\H$ is also an even function, we obtain a more definite non-existence result which holds not only in the space $\h2r$, but in $\s2r$ as well.

\begin{teo}\label{no1final}
If $\hipar$, then there do not exist properly embedded $\H$-surfaces in $\m2r$ with finite topology and one end.
\end{teo}

\begin{proof}
Suppose that the statement of Theorem \ref{no1final} does not hold, and let $\sig$ be a properly embedded $\Hs$ in $\m2r$ with finite topology and one end. Let $P_1,P_2$ be two horizontal planes at distance to each other greater than $2d$, where $d$ denotes as usual the diameter of the sphere in $\m2r$ with constant mean curvature equal to $H_0$, for some $H_0<\min\H$. After a vertical translation if necessary, the vertical height estimates for compact, embedded $\Hss$ given by Corollary \ref{estialturaverticalcompac} along with the plane separation lemma \ref{seplem} ensures us that $\sig$ is contained in the half-space determined by one of the horizontal planes $P_i$, say for definiteness $P_1$. A similar argument to the one used in the proof of Proposition \ref{1finalcilac} allows us to enclose $\sig$ in the open slab determined by the planes $P_1$ and $P_2$, i.e. the height function of $\sig$ is bounded.

However, this is a contradiction in both spaces $\s2r$ and $\h2r$, as detailed next. If the space is $\s2r$, then all the coordinates of $\sig$ are bounded, and since $\sig$ is proper then $\sig$ would be a closed surface, a contradiction with the fact that $\sig$ has one end. If the space is $\h2r$, then we know by Proposition \ref{1finalcilac} that $\sig$ is cylindrically bounded and thus is contained also in a compact region, which yields that $\sig$ is a closed surface, also a contradiction.

This concludes the proof of Theorem \ref{no1final}.
\end{proof}

The following structure-type result is a straightforward consequence of Theorem \ref{no1final}, Proposition 2.3 in \cite{Bue3} and the Hopf-type theorem in \cite{GaMi} for immersed spheres.

\begin{teo}\label{estructura1final}
Let be $\hipar$ and $\sig$ a properly embedded $\H$-surface in $\m2r$ with finite topology and at most one end. Then, $\sig$ is, up to ambient translations, a rotational $\H$-sphere.
\end{teo}

\begin{proofsk}
If $\sig$ has no ends, then $\sig$ is a closed, embedded surface and thus applying Alexandrov reflection technique ensures us that $\sig$ is topologically a sphere and that is rotational around some vertical line, see \cite{Bue3} for details on how this sphere is constructed. Moreover, $\sig$ is unique among all $\H$-spheres because of the uniqueness result stated in \cite{GaMi}. If $\sig$ has one end, then such surface does not exists because of Theorem \ref{no1final}.

This concludes the proof of Theorem \ref{estructura1final}.
\end{proofsk}

\def\refname{References}

\end{document}